\documentclass[12pt]{amsart}

\usepackage{CJKutf8}

\usepackage{amsmath,amssymb,latexsym, amsthm, amscd, mathrsfs, stmaryrd} %, txfonts}

\usepackage[mathscr]{euscript}

\usepackage{graphicx}

\usepackage{relsize}

\usepackage{pdfpages}

\usepackage{verbatim}
\usepackage{tikz}

\usepackage[Symbol]{upgreek}

\usepackage[all]{xy}
\usepackage{color}
\usepackage{hyperref}

\theoremstyle{definition}
\theoremstyle{plain}
\newtheorem{prop}[subsection]{Proposition}

\newtheorem{lem}[subsection]{Lemma}

\newcommand{\Hom}{\mathrm{Hom}}
\newcommand{\mbf}{\mathbf}
\newcommand{\mbb}{\mathbb}
\newcommand{\mrm}{\mathrm}

\newcommand{\C}{\mbb C}

\newcommand{\G}{\mathrm G}

\newcommand{\M}{\mathfrak M}

\newcommand{\bM}{\mbf M}

\newcommand{\bv}{\mbf v}
\newcommand{\bw}{\mbf w}
\newcommand{\bx}{\mbf x}

\renewcommand{\o}{\text{{\bf o}}}

\renewcommand{\i}{\text{{\bf i}}}
\newcommand{\ve}{\varepsilon}

\newcommand{\fS}{\mathfrak S}

\newcommand{\GL}{\mrm{GL}}

\newcommand{\w}{\upomega}

\newtheorem{thrm}{Theorem}

\title[Spaltenstein varieties of pure dimension]{Spaltenstein varieties of pure dimension
}

\author{Yiqiang Li}
\address{
Department of Mathematics\\
University at Buffalo\\
the State University of New York 
}
\email{yiqiang@buffalo.edu}


\begin{document}
\maketitle

\centerline{\it In memory of my uncle Renyi Huang}

%\centerline{
%\begin{CJK*}{UTF8}{gbsn}
%?????????????
%\end{CJK*}
%}
%\section{}
%\subsection{}

%\tableofcontents

\begin{abstract}
We show that Spaltenstein varieties of classical groups are pure dimensional when the Jordan type of 
the nilpotent element involved is an even or odd partition. 
We further show that they are Lagrangian in the  partial resolutions of the associated nilpotent Slodowy slices, from which 
their dimensions are known to be one half of the dimension of the partial resolution minus the dimension of the nilpotent orbit.
The results are then extended to the $\sigma$-quiver-variety setting. 
\end{abstract}

\section{Introduction}

\subsection{Spaltenstein varieties}
\label{Spaltenstein}

Let $G$ be a complex reductive group.
Fix a parabolic subgroup $P$ of $G$ and a nilpotent element $x$ in Lie$(G)$.
The Spaltenstein variety of the triple $(G, P, x)$ is defined to be
\[
X^P_x =\{ gP\in G/P | g^{-1}xg \in \mathfrak n_{\mrm{Lie} (P)}\}
\]
where 
$\mathfrak n_{\mrm{Lie} (P)}$ is the nilpotent radical of $\mrm{Lie}( P)$.
When $P$ is a Borel subgroup, a Spaltenstein variety
is  more commonly referred to as  a Springer fiber~\cite{Spr76}. 
%Spaltenstein varieties carry much representation-theoretic information and play a vital role in Springer theory~\cite{Spr76, BM83},
%category $O$~\cite{B08, W11} and invariants of tangles~\cite{Kh04}.
In general, a Spaltenstein variety is neither smooth nor irreducible. 
So an immediate question of substantial interest 
is if it is pure dimensional, 
%(a.k.a. equidimensional), 
that is if the dimensions of irreducible components of an $X^P_x$ are the same. 
It was answered in the affirmative in the following two fundamental cases by Spaltenstein~\cite{Sp76, Sp77} in the 70s,
and independently by Steinberg~\cite{St74}  for Case (a) when $G$ is a general linear group.
\begin{itemize}
\item[(a)]  $P$ is a Borel subgroup.
\item[(b)] $G$ is a general linear group.
\end{itemize}
%Steinberg~\cite{St74} also obtained (a) in the type-$A$ case independently.
%Vargars~\cite{V79} also obtained (a) in the type-$A$ case independently. 
Spaltenstein further
provided an example in~\cite[11.6]{Sp82}  showing that
the variety  $X^P_x$ is not always pure dimensional for a nilpotent element in $\mathfrak{so}_8$ of Jordan type $(1, 2^2, 3)$.
%Specifically,  when $G=\mrm{SO}_8(\mbb C)$,  $P$ is chosen
%such that $G/P$ is isomorphic to  the variety of isotropic subspaces $F_2\subseteq F_3$ in $\mbb C^8$ of
%dimension 2 and 3, respectively,  and $x$ is of Jordan type $(1, 2^2, 3)$,
%then $X^P_x$ is the union of  two irreducible components of respective dimensions $2$ and $3$.
This example is recalled in Section~\ref{Example}, together with 
a few more in {\it loc. cit.} where $X_x^P$ can be described explicitly with  fresh light casted upon.
Beyond Steinberg and Spaltenstein's results, little is known on the pure dimensionality of $X_x^P$.

In this paper, we shall prove
%present a third case where $X_x^P$ is pure dimensional as follows. 

\begin{thrm}
\label{Main-1}
The Spaltenstein variety $X^P_x$  is pure dimensional if
\begin{itemize}
\item[(c)] $G$ is classical and the Jordan type of  $x$ is an even or odd partition,
i.e., of the form
$1^{\bw_1} 3^{\bw_3}  5^{\bw_5} \cdots$ or $2^{\bw_2} 4^{\bw_4} 6^{\bw_6}\cdots$.
\end{itemize}
\end{thrm}

Our approach is to study Spaltenstein varieties in the context of  symplectic geometry. 

\subsection{Symplectic geometry and $\mbb C^*$-action}
As is gradually known, complex symplectic geometry provides 
a new and conceptual  way to understand the pure dimensionality of a complex variety. 
Precisely, there is the following remarkable result, whose proof can be found in the proof of Proposition 
%1.2.2 and 
5.4.7 in~\cite{G09}. 
Note that we work in the setting of complex algebraic geometry.

\begin{thrm}
%[\cite{G09} 5.4]
\label{symplectic}
Suppose that  $p: \widetilde Y  \to Y$ is a proper morphism from a  smooth symplectic  algebraic variety, with algebraic symplectic $2$-form,   to an affine variety. 
%Then any fiber of $p$ is isotropic. 
Suppose further that both varieties  admit a $\mbb C^*$-action, compatible with $p$. 
If  the following two conditions hold:
\begin{itemize}
\item the $\mbb C^*$-action provides a contraction of $Y$ to its fixed-point locus $Y^{\mbb C^*}$,
\item the $\mbb C^*$-action on $\widetilde Y$ has weight $1$ on the symplectic form $\omega$ on $\widetilde Y$:
$t^* \omega = t\omega$, $\forall t\in \mbb C^*$, 
\end{itemize}
then the fiber $p^{-1}(Y^{\mbb C^*})$, or rather its associated reduced scheme,  is Lagrangian in $\widetilde Y$. 
\end{thrm}

Being Lagrangian implies that $p^{-1}(Y^{\mbb C^*})$ is pure dimensional, provided that $\widetilde Y$ is so, 
and moreover its dimension is one half of that of $\widetilde Y$. 

It is exactly the framework of Theorem~\ref{symplectic} that Spaltenstein variety is put under and 
that the proof of  Theorem~\ref{Main-1} falls out, which we shall discuss  in more details as follows.

\subsection{Slodowy slices and their partial resolutions}
\label{Slodowy}
Retaining the setting in Section~\ref{Spaltenstein},
the cotangent bundle $T^* (G/P)$  of $G/P$ yields a partial resolution of singularities of 
the closure of a 
nilpotent orbit $\mathcal O_{e}$ in $\mrm{Lie} (G)$ for a Richardson element $e$:
\[
\pi'_P: T^* (G/P) \to \overline{\mathcal O}_{e}.
\]
Here the terminology `partial' refers to the fact that the restriction of $\pi'_p$ to the orbit $\mathcal O_e$ 
is generically finite, but not isomorphic, in general. 
When $P$ is a Borel,  the morphism $\pi'_P$ is the Springer resolution to the  nilcone of $G$ 
and a genuine resolution of singularities. 
On the other hand, fixing an $\mathfrak{sl}_2(\mbb C)$-triple $(x, y, h)$ in $\mrm{Lie}(G)$, one can consider 
the Slodowy slice $S_{x}: = x+ \ker \mrm{ad} (y)$ 
%\mathfrak z_y$, where $ \mathfrak z_y$ is the centralizer of $y$ in the Lie algebra $\mrm{Lie}(G)$ 
(see~\cite{Sl80}).
Setting $S_{e, x}= \overline{\mathcal O}_{e} \cap S_x$ and 
$\widetilde S_{e,x} = (\pi'_P)^{-1} ( S_{e, x})$ (so that $S_{e,x}$ is nonempty if and only if   $x\in \overline{\mathcal O}_e$). 
The above map $\pi'_P$ restricts to a partial resolution of the nilpotent Slodowy slice $S_{e,x}$:
\begin{align}
\label{pi}
\pi_P: \widetilde S_{e, x} \to S_{e, x},
\quad \mbox{with} \ \pi^{-1}_P (x)= X^P_x. 
\end{align}
Again when $P$ is a Borel, the morphism $\pi$ is a genuine resolution of singularities. 
The cotangent bundle $T^* (G/P)$ carries a canonical symplectic structure, i.e., a closed $2$-form, and 
from which the variety $\widetilde S_{e, x}$ inherits one, say $\omega$, as well. 
The variety $S_{e,x}$ is clearly an affine variety. 
%Thanks to the first statement in Theorem~\ref{symplectic}, 
%one immediately deduces
%that $X^P_x$ is isotropic in $\widetilde S_{e, x}$, which says that all irreducible components of $X^P_x$ have maximal possible dimension
%one half of that of $\widetilde S_{e,x}$.
%In particular
Thanks to~\cite[Corollary 3.5 b)]{BM83}, it is know that 
\begin{align}
\label{isotropic}
\dim X^P_x \leq \frac{1}{2} \dim \widetilde S_{e, x}.
\end{align}

In the cases (a) and (b) in Section~\ref{Spaltenstein}, the above inequality becomes equality and 
$X^P_x$ is Lagrangian in $\widetilde S_{e,x}$. 
We shall show that the same holds  for the case (c) in  Theorem~\ref{Main-1}.
Moreover,  $\widetilde S_{e, x}$ is  pure dimensional in general: it is a reduced complete intersection in $T^*(G/P)$ of  dimension
$\dim T^*(G/P) - \dim \mathcal O_x$ (see~\cite[Corollary 1.3.8]{G08}). 
Therefore we actually have a stronger version of Theorem~\ref{Main-1}.

\begin{thrm}
\label{Main-2}
If $G$ is a classical group and the
Jordan type of $x$  is an even or odd partition, then
the Spaltenstein variety $X^P_x$ is Lagrangian in $\widetilde S_{e, x}$ in (\ref{pi}), and hence
of pure dimension $\frac{1}{2} \dim T^*(G/P) - \frac{1}{2}\dim \mathcal O_x$.
\end{thrm}

With the above discussion, the proof of Theorem~\ref{Main-2} (and hence Theorem~\ref{Main-1}) finally boils down 
to a search of the desired $\mbb C^*$-actions for $\widetilde S_{e,x}$ and $S_{e, x}$ to apply Theorem~\ref{symplectic}.  

Both varieties $\widetilde S_{e, x}$ and  $S_{e, x}$ admit a natural $\mbb C^*$-action induced from the $\mathfrak{sl}_2(\mbb C)$-triple $(x, y, h)$ so that $\pi_P$ is $\mbb C^*$-equivariant.
Moreover the $\mbb C^*$-action provides a contraction of $S_{e, x}$ to $\{x\}$, its $\mbb C^*$-fixed point (\cite[1.4]{G08}). 
However, the $\mbb C^*$-action on the symplectic structure $\omega$ has weight $2$, 
instead of weight $1$ as required in Theorem~\ref{symplectic}. 
This defect  is expected in light of Spaltenstein's example: 
there is no uniform $\mbb C^*$-action on $\widetilde S_{e,x}$ and $S_{e,x}$ for all $e$ and $x$ satisfying all conditions in 
Theorem~\ref{symplectic}.

Instead we  obtain the desired $\mbb C^*$-actions in the setting of 
Nakajima quiver varieties~\cite{N94, N98} and their variants in~\cite{Li18}, from which this paper is grown out.
%This is the reason why we have to assume that $G$ is a classical group. 
%A possible case that the above observation can not detect is when $X^P_x$ is equidimensional, but not Lagranian. 
%However, whether such case appears is not clear to the author.  

\subsection{$\mbb C^*$-action on Nakajima varieties}

Thanks to the works of Nakajima~\cite{N94} and Maffei~\cite{M05}, 
the proper map $\pi_P$ for $G$ being a general linear  group has an incarnation as Nakajima quiver varieties attached
to a type-$A$ quiver.
%  whose underlying graph  is a Dynkin graph of type $A$:
\begin{align}
\label{piA1}
\pi_A: \M_{\zeta}(\bv, \bw)_{A} \to \M_1(\bv,\bw)_A.
\end{align}
Here $\bv$ and $\bw$ are tuples of integers determined by the Jordan types of the Richardson element $e$ and $x$ respectively,
and $\zeta$ is a generic parameter used for the stability condition. 
The orientation induces intrinsically a $\mbb C^*$-action on the quiver varieties 
$\M_{\zeta}(\bv, \bw)_{A}$ and $\M_1(\bv,\bw)_A$.
This action satisfies all conditions in Theorem~\ref{symplectic} and hence provides a conceptual proof of 
the pure dimensionality of $X_x^P$ for $G$ being a general linear group, i.e., Case~\ref{Spaltenstein}(b). 

If $G$ is classical, i.e., an orthogonal or symplectic group, 
the map $\pi_P$ admits a quiver description $\pi_{\sigma, A}$, as a restriction of $\pi_A$,
in the recent work~\cite{Li18}, with 
$\widetilde S_{e,x}$ and $S_{e,x}$ realized as the fixed-point 
loci $\fS_{\zeta}(\bv,\bw)_A$ (resp. $\fS_1(\bv,\bw)_A$) 
of Nakajima varieties $\M_{\zeta}(\bv, \bw)_{A}$  (resp. $\M_1(\bv,\bw)_A$)  under a specific  involution $\sigma$:
\begin{align}
\label{piA2}
\pi_{\sigma,A}: \fS_{\zeta}(\bv, \bw)_{A} \to \fS_1(\bv,\bw)_A.
\end{align}
The $\mbb C^*$-actions on Nakajima varieties can not be compatible with the involution in general, again due to
Spaltenstein's example. 
The crucial observation is that the place where the $\mbb C^*$-action and the involution $\sigma$ is compatible is 
where $X^P_x$ is Lagrangian.  
To this end, we show that
the tuple $\bw$ under the conditions in Theorem~\ref{Main-2} are the compatible places
for the $\mbb C^*$-action and the involution, hence providing a proof of Theorems~\ref{Main-2} and~\ref{Main-1} finally.

The arguments are indeed not restricted to type-$A$ graphs. We are able to establish a result that is valid for 
all Dynkin graphs. 
We drop the subscript $A$ in (\ref{piA1}) and (\ref{piA2}) to denote the morphism between
Nakajima varieties of a fixed Dynkin graph.

\begin{thrm}
\label{Main-3}
Assume that $\bw_i\bw_j=0$ if there is an edge joining $i$ and $j$. 
Then the fiber of the $\mbb C^*$-fixed point under $\pi_{\sigma}$ is Lagrangian 
in $\fS_{\zeta}(\bv, \bw)$. 
\end{thrm}

The main content of the paper is to study the compatibility of the $\mbb C^*$-action and the automorphism $\sigma$ in order to 
prove Theorem~\ref{Main-3}. 
When the signature $c^0$ of the diagram isomorphism in the automorphism $\sigma$ is $-1$, we can drop the assumption on $\bw$ in Theorem~\ref{Main-3} and 
this more general result is stated in Theorem~\ref{Main-4}. 

\subsection{Layout of the paper}
In Section ~\ref{pre}, we recall Nakajima varieties and their $\sigma$ variants.
In Section~\ref{new}, we study the compatibility of $\mbb C^*$-action with the various isomorphisms in the definition of $\sigma$-quiver varieties.
In Section~\ref{Example}, we reproduce  Spaltenstein's examples in~\cite[11.6, 11.8]{Sp82} with new observations on being Lagrangian.  

\subsection{Acknowledgements}
We thank the anonymous referee for helpful comments and insightful suggestions. 
This work was  partly supported by the NSF grant DMS 1801915. 
%\setcounter{tocdepth}{1}
%\tableofcontents

\subsection{Updates after publication}

(1).
Theorem~\ref{Main-3} implies the following result stronger than Theorem~\ref{Main-1}: the variety $X_x^P$ is pure dimensional if  
\begin{itemize}
\item[(c')] $G$ is classical and the Jordan type, say
$1^{\bw_1} 2^{\bw_2} 3^{\bw_3}   \cdots$, of $x$ satisfies  $\bw_i \bw_{i+1}=0$ for all $i$. 
\end{itemize}

I thank Elek Balazs for pointing this out to me.

(2). A nilpotent orbit whose partition is either purely even or purely odd is called an even orbit in literature. 
I thank Bingyong Sun for pointing this out to me.

(3). The variety $X_x^P$ is invariant under row reduction and is conjectured to be true under column reduction. 
See the arXiv paper \href{https://arxiv.org/abs/2002.04422}{\color{blue}{arXiv:2002.04422}}, Remark 2.4.3.

\section{Preliminaries on quiver varieties}
\label{pre}

In the section, we recall briefly Nakajima varieties~\cite{N94, N98} and their $\sigma$ variants in~\cite{Li18}. 
Our treatment follows closely Sections 1-4 in ~\cite{Li18}.

\subsection{Nakajima varieties}
\label{lavw}

Let $\Gamma$ be a Dynkin graph. Let
$I$ and $H$ be  the vertex and arrow set, respectively.
For each arrow $h$, let $\o (h)$ and $\i (h)$ be its outgoing and incoming vertex.
There is an involution on the arrow set $\bar \empty: H \to H$, $h \mapsto \bar h$ such that
$\o (\bar h) = \i (h) $ and $\i (\bar h) = \o (h)$.
Let $V=\oplus_{i\in I} V_i$ and $W=\oplus_{i\in I}W_i$ be two finite dimensional $I$-graded 
vector spaces over the complex field $\mbb C$
of dimension vectors $\bv = (\bv_i)_{i\in I}$ and $\bw =(\bw_i)_{i\in I}$, respectively.
The framed representation space of the graph $\Gamma$ in $V\oplus W$ is 
\begin{align}
\label{bM}
\bM(\bv,\bw)
=\oplus_{h\in H} \Hom (V_{\o (h)}, V_{\i (h)}) \oplus \oplus_{i\in I} \Hom(W_i, V_i) \oplus  \Hom (V_i, W_i).
\end{align}
When $V$ and $W$ shall be highlighted, we  write $\bM(V, W)$ for $\bM(\bv, \bw)$. 
An element in $\bM(\bv, \bw)$ is denoted by $\bx \equiv (x, p, q) \equiv (x_h, p_i, q_i)_{h\in H, i\in I}$ where
$x_h$ is in $\Hom (V_{\o (h)}, V_{\i(h)})$, $p_i$ in $\Hom (W_i, V_i)$ and $q_i$ in $\Hom (V_i, W_i)$.
Let $\ve^0 : H\to \{ \pm 1\}$ be an orientation function such that $\ve^0(h)+\ve^0(\bar h) =0$, $\forall h\in H$.
To a point $\bx \in \bM(\bv, \bw)$, we set 
\begin{align}
\label{a-b}
a_i (\bx) = (q_i, x_h)_{h: \o (h) = i}  \quad
\mbox{and} \quad b_i (\bx) = (p_i, \ve^0(\bar h) x_h)_{h: \i (h) = i}.
\end{align}
The space $\bM(\bv, \bw)$ admits  a symplectic structure with respect to $\ve^0$ given by
\begin{align}
\label{symplectic-form}
\omega (\bx, \bx') =  \mrm{trace} \left (\sum_{i\in I}  b_i(\bx) a_i(\bx') - q_i p'_i  \right ), 
%\sum_{h\in H} \mrm{tr} ( \ve^0(h) x_h x'_{\bar h}) + \sum_{i\in I} \mrm{tr} (p_i q'_i - p'_i q_i),
\quad \forall \bx, \bx' \in \bM(\bv, \bw).
\end{align}
Let 
$\G_{\bv}   = \prod_{i\in I} \GL (V_i)$
act on $\bM(\bv, \bw)$ from the left as follows.
For all $g=(g_i)_{i\in I} \in \G_{\bv}$ and $\bx \in \bM(\bv, \bw)$, we define
$g. \bx = \bx' \equiv (x'_h, p'_i, q'_i)$ where
$x'_h = g_{\i (h)} x_h g^{-1}_{\o (h)}$, $p'_i = g_i p_i $ and $q'_i = q_i g^{-1}_i$ for all $h\in H$ and $i\in I$.
Let 
\[
\mu_{\mbb C} : \bM(\bv, \bw) \to \mrm{Lie} (\G_{\bv})^* 
\]
be the moment map associated to the $\G_{\bv}$-action on the symplectic space $\bM(\bv, \bw)$. 
After identifying $\mrm{Lie} (\G_{\bv}) =\oplus_{i\in I} \mathfrak{gl}(V_i)$ with its dual  $\mrm{Lie} (\G_{\bv})^* $ via the trace form, 
the  $i$-th component of $\mu_{\mbb C}$  is given by
$\mu^{(i)}_{\mbb C}(\bx) = b_i(\bx) a_i(\bx)$.

Let $[\mbf x]$ denote the $\G_{\bv}$-orbit of $\mbf x$ in $\bM(\bv, \bw)$.
%When all components are 0, its $\G_{\bv}$-orbit is  denoted by $[0]$. 

Fix an embedding $\mbb C^I \to \mrm{Lie}(\G_\bv)$ by 
$(\zeta^{(i)}_{\mbb C})_{i\in I}  \mapsto (\zeta^{(i)}_{\mbb C} \mrm{Id}_{V_i})_{i\in I}$
for all $\zeta_{\mbb C} =(\zeta^{(i)}_{\mbb C})_{i\in I} \in \mbb C^I$.
Let
$\Lambda_{\zeta_{\mbb C}} (\bv, \bw)$ be the fiber $\mu^{-1}_{\mbb C}(\zeta_{\mbb C})$.
The group $\G_{\bv}$ acts on $\Lambda_{\zeta_{\mbb C}}(\bv,\bw)$. 

Let $\xi =(\xi_i)_{i\in I} \in \mbb Z^I$.
Fix an element $x =(x_h)_{h\in H}$  in the first component of  $\bM(\bv, \bw)$ and an $I$-graded subspace $S=(S_i)_{i\in }$ of $V$, 
we say that $S$ is $x$-invariant if $x_h( S_{\o (h)} ) \subseteq S_{\i(h)}$ for all $h\in H$.
A point $\bx =(x, p, q)$ in $\bM(\bv, \bw)$ is called $\xi$-semistable  
if the following two  stability conditions are satisfied.
For any  $I$-graded subspaces $S$ and $T$ of $V$  of dimension $\mbf s$ and $\mbf t$, respectively, 
\begin{align}
& \text{
if $S$ is $x$-invariant and $S\subseteq \ker q$, then $\xi \cdot \mbf s \leq 0$,
}
\tag{S1}
\\
& \text{
if $T$ is $x$-invariant and $T \supseteq \mrm{im} \ p$, then
$\xi \cdot \mbf t \leq \xi \cdot \bv$.
}
\tag{S2}
\end{align}
Let $\Lambda^{\xi\text{-}ss}_{\zeta_{\mbb C}} (\bv, \bw)$ be the  $\G_{\bv}$-invariant set of all $\xi$-semistable points in 
$\Lambda_{\zeta_{\mbb C}}(\bv, \bw)$.

Let $\mbf C=(c_{ij})_{i, j\in I}$ is the Cartan matrix of the graph $\Gamma$. We set
\begin{align*}
\begin{split}
R_+  =\{ \gamma \in \mbb N^I | \ ^t\gamma\mbf C \gamma \leq 2\} \backslash \{ 0\},
R_+(\bv)  =\{\gamma \in R_+ | \gamma_i \leq \bv_i, \forall i\in I\},
D_{\gamma}  =\{ a\in \mbb C^I | a\cdot \gamma=0\}.
\end{split}
\end{align*}
A parameter $\zeta=(\xi, \zeta_{\mbb C})\in \mbb Z^I \times \mbb C^I$ is called generic 
if it satisfies 
$\xi \in \mbb Z^I \backslash  \cup_{\gamma\in R_+(\bv)} D_{\gamma}$
or
$\zeta_{\mbb C} \in \mbb C^I \backslash  \cup_{\gamma\in R_+(\bv)} D_{\gamma}$.
From now on, we assume that $\zeta$ is generic. 
When $\zeta$ is generic, the group $\G_{\bv}$ acts freely on $\Lambda^{\xi\text{-}ss}_{\zeta_{\mbb C}}(\bv, \bw)$.
Following Nakajima ~\cite{N94, N98}, we define the quiver variety
attached to the data $(\Gamma, \ve^0, \bv,\bw, \zeta)$
to be
\begin{align}
\label{Nakajima}
\M_{\zeta} (\bv, \bw) = \Lambda^{\xi\text{-}ss}_{\zeta_{\mbb C}}(\bv, \bw)/\G_{\bv},\quad 
\zeta\equiv (\xi, \zeta_{\mbb C}) \in \mbb Z^I \times \mbb C^I \ \mbox{generic}.
\end{align}
Let $\M_0(\bv,\bw)$ be the affinization of $\M_{\zeta}(\bv,\bw)$, with which is equipped 
a projective morphism 
$\pi: \M_{\zeta} (\bv, \bw) \to \M_0 (\bv, \bw).$
Let $\M_1(\bv,\bw)$ be the image of $\M_{\zeta}(\bv,\bw)$ under $\pi$ 
so that $\pi$ factors through a proper
map under the same notation, which is (\ref{piA1}) in type $A$:
\begin{align}
\label{pi}
\pi: \M_{\zeta} (\bv, \bw) \to \M_1 (\bv, \bw).
\end{align}
The variety $\M_{\zeta}(\bv,\bw)$ is smooth  and symplectic with  the latter  induced from $\bM(\bv, \bw)$.

\subsection{$\sigma$-quiver varieties}
In this section, we recall $\sigma$-quiver varieties from~\cite{Li18}.

\subsubsection{Reflection functors}
Recall the Cartan matrix $\mbf C=(c_{ij})$.
For each $i\in I$, we define a bijection $s_i: \mbb Z^I \to \mbb Z^I$ by
$s_i (\xi) = \xi'$ where $\xi'_j = \xi_j  - c_{ji} \xi_i $, $\xi = (\xi_j)_{j\in I}$, $\xi'=(\xi_j')_{j\in I} \in \mbb Z^I$.
Let $\mathcal W$  be the Weyl group generated by $s_i$ for all $i\in I$. 

Let $s_i*_{\bw}\bv$ denote the vector whose $j$-component is $\bv_j$ if $j\neq i$ and whose $i$-th component is 
$\bw_i + \sum_{h: \o (h)=i} \bv_{\i (h)} -\bv_i$. 

The reflection  functor $S_i$ of Nakajima, Lusztig and Maffei~\cite{L00, M02, N03} associated to  the simple reflection $s_i$ is defined to be
\begin{align*}
%\label{Si}
S_i : \M_{\zeta} (\bv, \bw) \to \M_{s_i (\zeta)} (s_i *_\bw \bv, \bw),  [\bx] \mapsto [\bx'], \quad \mbox{if} \ \xi_i < 0 \ \mbox{or}\ \zeta^{(i)}_{\mbb C} \neq 0,
\end{align*}
where the pair ($[\bx], [\bx']$) satisfies the conditions (R1)-(R4) as follows. 
Let 
$V'$ be a  vector space of dimension $ s_i *_\bw \bv$ 
such that  $V_j' = V_j$ if $j\neq i$ and $U_i = W_i \oplus \oplus_{h\in H : \o (h) =i} V_{\i (h)}$.
%and $V_i'= U_i/V_i$. 
\begin{align}
& 
\begin{CD}
0 @>>> V_i'  @>a_i(\bx')>>  U_i @> b_i(\bx) >>  V_i  @>>>  0 \quad \mbox{is exact}, 
\end{CD}
\tag{R1}
\\  
& a_i(\bx) b_i(\bx) - a_i(\bx') b_i(\bx') = \zeta'^{(i)}_{\mbb C},  \quad \zeta'_{\mbb C} = s_i (\zeta_{\mbb C}),  
\tag{R2}
\\
& x_{h} = x'_{h}, p_j = p'_j, q_j=q'_j, \quad  \hspace{1.2cm} \mbox{if} \ \o (h) \neq i, \i (h) \neq i \ \mbox{and} \ j\neq i,  \tag{R3} \\
& \mu_j (\bx) = \zeta^{(j)}_{\mbb C},
% \lambda_j,
\mu_j (\bx') = \zeta'^{(j)}_{\mbb C},
% \lambda_j', 
\quad \hspace{.85cm} \mbox{if}  \ j\neq i.   
\tag{R4}
\end{align}

Since $(s_i (\xi))_i >0$ if $\xi_i < 0$, we can define the reflection $S_i$ when $\xi_i >0$, by switching the roles of $\bx$ and $\bx'$.
So if $\w=s_{i_1}s_{i_2}\cdots s_{i_l} \in \mathcal W$ and $\zeta$ is generic, 
the reflection functor $S_{\w}$ is defined to be the composition of the $S_i$'s:
\begin{align}
\label{sw}
S_{\w} = S_{i_1} S_{i_2} \cdots S_{i_l} : \M_{\zeta}(\bv, \bw) \to \M_{\w(\zeta)} (\w *_\bw \bv, \bw),
\end{align}
where $\w*_\bw \bv$ is a composition of $s_{i_j}*_\bw\bv$'s.

\subsubsection{The transpose $\tau$}
\label{sub-tau}

To any linear transformation $T: E\to E'$ between two vector  spaces, each equipped with a non-degenerate bilinear form $(-, -)_E$ and $(-,-)_{E'}$, 
we define  its right adjoint $T^*: E' \to E$ by the rule 
\[
(T(e), e')_{E'} = (e, T^* (e'))_{E}, \quad \forall e \in E, e' \in E'. 
\]
There is  an isomorphism $\Hom (E, E') \cong \Hom (E', E)$ defined by $T\mapsto T^*$.

Assume that the $i$-th components $V_i$ and $W_i$ of  $V$ and $W$   are equipped with non-degenerate bilinear forms for all $i\in I$. 
We define an automorphism
\begin{align*}
\tau : \bM(\bv, \bw) \to \bM(\bv, \bw), \quad \bx=(x_h, p_i, q_i) \mapsto {}^{\tau} \bx =({}^{\tau} x_h, {}^{\tau} p_i, {}^{\tau}q_i)
\end{align*}
where  ${}^\tau x_h = \ve (h)  x_{\bar h}^*$, ${}^{\tau} p_i = -  q_i^{*}$ and ${}^{\tau} q_i =  p_i^*$ 
for all $h\in H$ and $i\in I$.
This automorphism induces an isomorphism:
\begin{align}
\label{tau}
\tau: \M_{\zeta}(\bv, \bw) \to \M_{-\zeta}(\bv,\bw). 
\end{align}

\subsubsection{Diagram isomorphism $a$}

Let $a$ be an automorphism of $\Gamma$, i.e.,  there are automorphisms  of vertex and arrow sets, both denoted by $a$, such that
$a(\o(h))= \o (a(h))$, $a(\i (h)) = \i (a(h))$ and $a(\bar h) = \overline{a(h)}$ for all $h\in H$. 
Assume that $a$ is compatible with the  function $\ve^0$ in the following sense. 
There exists a constant $c^0\equiv c_{a, \ve^0} \in \{ \pm 1\}$ such that 
\begin{align}
\label{compatible-pair}
\ve^0(a(h)) = c^0 \cdot \ve^0(h), \  \forall h\in H .
\end{align}
Let $a(V)$ be the $I$-graded vector space whose $i$-th component is $V_{a^{-1}(i)}$.
The dimension vector of $a(V)$ is  $a(\bv)$  whose $i$-entry is  $\bv_{a^{-1}(i)}$. 
Given any point $\bx=(x, p, q) \in  \bM(V, W)$, 
we define a point $a(\bx)=(a(x), a(p), a(q) )\in \bM(a(V), a(W))$
by
\[
a(p)_i=p_{a^{-1}(i)}, \ a(q)_i = q_{a^{-1}(i)}, \ 
a(x)_h= \ve^0(h)^{\frac{1-c^0}{2}}  x_{a^{-1}(h)}, \quad \forall i\in I, h\in H.
\]
It induces a diagram isomorphism on Nakajima varieties:
\begin{align}
\label{dia}
a: \M_{\zeta} (\bv, \bw) \to \M_{a(\zeta)} (a(\bv), a(\bw)).
\end{align}

\subsubsection{$\sigma$-Quiver varieties}

%In this subsection, $\zeta$ is assumed to be generic. 
Consider 
\begin{align}
\label{sigma}
\sigma: =a  S_{\w} \tau : \M_{\zeta}(\bv, \bw) \to \M_{- a \w(\zeta)} (a(\w *_\bw \bv), a(\bw)),
\end{align}
where $\tau$, $S_{\w}$ and $a$ are in (\ref{tau}), (\ref{sw}) and (\ref{dia}), respectively.
The $\sigma$-quiver variety is defined by
\begin{align}
\begin{split}
\fS_\zeta(\bv,\bw) = (\M_{\zeta}(\bv,\bw))^{\sigma}, \quad  \fS_1(\bv,\bw) = \pi (\fS_\zeta(\bv,\bw)),  \\ \mbox{if}\
\bw = a(\bw), \zeta= - a \w (\zeta), \bv= a(\w*_\bw \bv).
\end{split}
\end{align}
The proper map $\pi$ restricts to a proper morphism which is (\ref{piA2}) in type $A$:
\begin{align}
\pi_{\sigma} : \fS_{\zeta}(\bv,\bw) \to \fS_1(\bv, \bw).
\end{align}
$\fS_\zeta(\bv,\bw)$ has a symplectic structure  inherited from that of $\M_{\zeta}(\bv,\bw)$ and $\fS_1(\bv,\bw)$ is an affine variety as a closed subvariety of $\M_1(\bv,\bw)$.

For the rest of this section, we consider the Dynkin graph  of type $A_n$:
$1 \leftrightarrows 2 \leftrightarrows \cdots \leftrightarrows n$.
Set $\ve^0 (h) = i -j $ if $h$ is an arrow from $i$ to $j$ and $c^0=1$. 
The automorphism $a$ is the identity automorphism. 
The Weyl group element $\w$ is the longest Weyl group element. 
Let $\zeta = (\xi, 0)$ where all components in $\xi$ is $1$.
%In this case, we index the vertex set as $I=\{1, \cdots, n\}$. 
For any pair $(\bv, \bw)$,
we define a new pair  $(\widetilde \bv=(\widetilde \bv_i)_{1\leq i\leq n}, \widetilde \bw=(\widetilde \bw_i)_{1\leq i\leq n})$
where
\begin{align}
\label{tild-bw}
\widetilde \bv_i = \bv_i + \sum_{j\geq i+1} (j-i) \bw_j, \quad 
\widetilde \bw_i = \delta_{i, 1} \sum_{1\leq j \leq n} j \bw_j,\quad \forall 1\leq i\leq n.
\end{align}
Now set
$
\mu = ( \widetilde \bv_0 - \widetilde \bv_1, \widetilde \bv_1 - \widetilde \bv_2, \cdots, 
\widetilde \bv_{n-1} - \widetilde \bv_n, \widetilde \bv_n).
$
Let $P_{\mu}$ be a parabolic subgroup of a classical group $G$ whose levi has 
size indexed by $\mu$. 
In other words, the isotropic flag variety  $G/P_{\mu}$ is the collection of  all isotropic flags such
that the dimension difference of the $i$-th step flag and $(i+1)$-th step flag is $\widetilde \bv_{i-1} -\widetilde \bv_i$.
Note that $P_{\mu}$ may be empty.  Let $e_{P_\mu}$ be the associated Richardson element. 
%We have $\mu_i = \bw_i + \cdots + \bw_n - \bv_i + \bv_{i-1}$.
%Reorder the entries in $\mu$ in decreasing order: $\rho_1\geq \rho_2 \geq \rho_3 \geq \cdots \geq \rho_{n+1}$ and set
%$\mu' = 1^{\rho_1 -\rho_2} 2^{\rho_2-\rho_3} \cdots n^{\rho_n -\rho_{n+1}} (n+1)^{\rho_{n+1}}$.
Let $$\lambda = 1^{\bw_1} 2^{\bw_2}\cdots.$$ 
We write  $\tilde S_{e_{P_\mu},x}$ in Section~\ref{Slodowy} as $\tilde S^{\mrm{Lie}(G)}_{e_{P_\mu}, \lambda}$ when the Jordan type of $x$ is $\lambda$. 
The following result is obtained in~\cite[Corollary 8.3.4]{Li18}.

\begin{prop}
\label{8.3.4}
\begin{enumerate}
\item If $W_i$ is equipped with a symmetric (resp$.$ skew-symmetric) form  for $i$ even (resp$.$ odd), then 
$\fS_{\zeta} (\bv, \bw)  \cong 
\widetilde{\mathcal S}_{e_{P_\mu}, \lambda}^{\mathfrak o_{\widetilde \bw_1}}$ 
and $\fS_{1}(\bv, \bw) \cong  \mathcal S_{e_{P_\mu}, \lambda}^{\mathfrak o_{\widetilde \bw_1}}$.
\item 
If forms on $W_i$ are skew-symmetric (resp$.$ symmetric) for $i$ even (resp$.$ odd), then
$\fS_{\zeta} (\bv, \bw) \cong 
\widetilde{\mathcal S}_{e_{P_\mu}, \lambda}^{\mathfrak{sp}_{\widetilde \bw_1}}$ 
and 
$\fS_{1}(\bv, \bw) \cong  \mathcal S_{e_{P_\mu},\lambda}^{\mathfrak{sp}_{\widetilde \bw_1}}$. 
\end{enumerate}
\end{prop}

\section{$\mbb C^*$-action and the automorphism $\sigma$}

\label{new}

In this section we assume that $\zeta$ is generic and $\zeta_{\mbb C}=0$.
We study the compatibility of modified version of a $\mbb C^*$-action in~\cite[Section 5]{N94}  with the automorphism $\sigma$.
By using these analyses, we then provide proofs for Theorems~\ref{Main-1}-\ref{Main-3}.

\subsection{Compatibility}
To an orientation $\ve$ of $H$, not necessarily the same as $\ve^0$ in the definition of Nakajima varieties, we can define two $\C^*$-actions on $\bM(\bv,\bw)$ in (\ref{bM}).
The first one   is given by  $(t, \bx) \mapsto t \circ_{\ve} \bx$  where 
\begin{align}
\label{t1}
t \circ_{\ve} \bx = ( t^{\frac{1 + \ve (h)}{2}}  x_h, p_i, t q_i).
\end{align}
The second one is given by $(t,\bx) \mapsto t\star_{\ve}\bx$ where 
\begin{align}
\label{t2}
t \star_{\ve} \bx = (t^{\frac{1 + \ve (h)}{2}}  x_h, tp_i,  q_i).
\end{align}

It is clear that each $\mbb C^*$-action induces a $\mbb C^*$-action on $\M_{\zeta}(\bv,\bw)$ in (\ref{Nakajima}), in light of the assumption that $\zeta_{\mbb C}=0$,  but 
 the induced  ones on $\M_{\zeta}(\bv, \bw)$ coincide as follows so that we do not have to distinguish the two actions
 on $\M_{\zeta}(\bv,\bw)$.
 
\begin{lem}
\label{two-one}
We have $t\circ_{\ve} [\bx] = t \star_{\ve} [\bx]$ for all $[\bx]\in \M_{\zeta}(\bv,\bw)$.
\end{lem} 

\begin{proof}
Let $g=(t.\mrm{id}_{V_i})_{i\in I}$. Then
$g. (t\circ_{\ve} \bx) =t \star_{\ve} \bx$ as required. 
\end{proof}

It is clear that the weight of the symplectic form on $\M_{\zeta}(\bv,\bw)$ with respect to this $\mbb C^*$-action is $1$, i.e.,
$\omega(t\circ_{\ve} [\bx], t\circ_{\ve}[\bx']) = t \omega([\bx],[\bx'])$.
 Since the graph is Dynkin,  the $\mbb C^*$-action provides a contraction from 
 $\M_0(\bv,\bw)$, and hence $\M_1(\bv,\bw)$,  to its $\mbb C^*$-fixed point $[0]$.

The following lemma is the compatibility of the transpose $\tau$ in Subsection~\ref{sub-tau} and the $\mbb C^*$-action.

\begin{lem}
\label{t-T}
We have $\tau ( t \circ_{\ve} [\bx]) = t \circ_{-\ve} \tau ([\bx])$ for all $t\in \mbb C^*$ and $[\bx]\in \M_{\zeta}(\bv, \bw)$. 
\end{lem}

\begin{proof}
We write $t\circ_{\ve} \bx= (t\circ_{\ve} x_h, t\circ_{\ve}p_i, t\circ_{\ve} q_i)_{h\in H, i\in I}$ and
$\tau(t\circ_{\ve}[\bx])= [(x'_h, p'_i, q'_i)]$. We have
\begin{align*}
x'_h &= \ve(h) (t\circ_{\ve} x_{\bar h})^* = \ve (h). ( t^{\frac{1+ \ve (\bar h)}{2}} x_{\bar h})^* = t^{\frac{1+ \ve (\bar h)}{2}}  \tau (x_h) ,\\
p'_i & = - (t\circ_{\ve} q_i)^* = -( tq_i)^*= t \tau(p_i),\\
q'_i & = (t\circ_{\ve} p_i)^*= p^*_i = \tau (q_i).
\end{align*}
This shows that $(x'_h, p'_i, q'_i)=  t \star_{-\ve} \tau (\bx)$, and the  Lemma follows readily by Lemma~\ref{two-one}.
\end{proof}

Let $a$ be an automorphism of $\Gamma$. 
We assume that the pair $(a,\ve)$ is compatible with signature $c \in \{\pm 1\}$, see (\ref{compatible-pair}). 
%if  $\ve (a(h)) = c \ve(h)$ for $h\in H$. 
We have the following compatibility of the automorphism $a$ and the $\mbb C^*$-action.

\begin{lem}
\label{t-a}
Let $(a,\ve)$ be a compatible pair with signature $c$. 
Then $a(t \circ_{\ve} [\bx]) = t \circ_{c\ve} a([\bx])$, for all $t\in \mbb C^*$ and $[\bx]\in \M_{\zeta}(\bv, \bw)$.
\end{lem}

\begin{proof}
We write $t\circ_{\ve} \bx= (t\circ_{\ve} x_h, t\circ_{\ve}p_i, t\circ_{\ve} q_i)_{h\in H, i\in I}$ and
$a(t\circ_{\ve}\bx)= (x'_h, p'_i, q'_i)$. We have
\begin{align*}
x'_h & = \ve(h)^{\frac{1-c}{2}} t\circ_{\ve} x_{a^{-1}(h)} = \ve(h)^{\frac{1-c}{2}} t^{\frac{1+ \ve(a^{-1}(h))}{2}} x_{a^{-1}(h)} 
= t^{\frac{1+ c \ve(h)}{2}} a(x)_h,\\
p'_i  & = t\circ_{\ve} p_{a^{-1}(i)} = p_{a^{-1}(i)} = a(p)_i,\\
q'_i & = t \circ_{\ve} q_{a^{-1}(i)} = t . q_{a^{-1}(i)}  = t .a(q)_i.
\end{align*}
So $(x'_h, p'_i, q'_i)=  t \star_{-\ve} a (\bx)$. The Lemma follows. 
\end{proof}

The following lemma is the compatibility of the reflection functor $S_i$ and the $\mbb C^*$-action.

\begin{lem}
\label{t-S}
We have $S_i (t\circ_{\ve}[\bx]) = t\circ_{\ve} S_i([\bx])$ for all $t\in \mbb C^*$ and $[\bx]\in \M_{\zeta}(\bv,\bw)$. 
\end{lem}

\begin{proof}
Let $S_i([\bx])=[\bx']$. It suffices to show that the pair $(t\circ_{\ve}\bx, t\circ_{\ve} \bx')$ satisfies the conditions (R1)-(R4) 
in the definition of reflection functors. 
Recall $a_i(\bx)$ and $b_i(\bx)$ from (\ref{a-b}).  
% that 
%\[
%a_i(\bx) = (q_i, x_h)_{h: \o (h) =i},  b_i(\bx) = (p_i, \ve (\bar h) x_h)_{h: \i (h) =i} .
%\]
There is 
\[
a_i( t\circ_{\ve}\bx) = (t q_i, t^{\frac{1+\ve(h)}{2}} x_h)_{h: \o (h) =i} , 
b_i(t\circ_{\ve}\bx) = (p_i, t^{\frac{1+\ve(h)}{2}} \ve^0(\bar h) x_h)_{h: \i(h)=i}.
\]
Thus we must have
\[
b_i(t\circ_{\ve} \bx) a_i(t\circ_{\ve} \bx') = t b_i(\bx) a_i(\bx') =0, 
\]
Clearly, $b_i(t\circ_{\ve}\bx)$ is surjective since $b_i(\bx)$ is so and 
$a_i(t\circ_{\ve}\bx')$ is injective since $a_i(\bx')$ is so.
Hence (R1) holds for the pair $(t\circ_{\ve}\bx, t\circ_{\ve} \bx')$.
Similarly, there is 
\[
a_i(t\circ_{\ve} \bx) b_i(t\circ_{\ve}\bx) - a_i(t\circ_{\ve} \bx') b_i(t\circ_{\ve} \bx') = t \left (a_i(\bx) b_i(\bx) - a_i(\bx')  b_i(\bx') \right ) =0.
\]
This shows that the pair $(t\circ_{\ve}\bx, t\circ_{\ve} \bx')$ satisfies (R2). 
The condition (R3) for $(t\circ_{\ve}\bx, t\circ_{\ve} \bx')$ is clearly followed from definition.
The condition (R4) for $(t\circ_{\ve}\bx, t\circ_{\ve} \bx')$ can be proved in a similar way as that of (R2). 
The Lemma thus follows. 
\end{proof}

By combining Lemmas~\ref{t-T}, \ref{t-a} and \ref{t-S}, we have the following proposition.

\begin{prop}
\label{t-sigma}
Let $(a, \ve)$ be a compatible pair with signature $c$.
Then we have 
$$\sigma (t \circ_{\ve} [\bx]) = t \circ_{- c \ve} \sigma([\bx]), \forall t\in \mbb C^*, [\bx]\in \M_{\zeta}(\bv, \bw).$$
%and $\sigma(t\star_{\ve} [\bx]) = t\circ_{-c\ve} \sigma([\bx])$.
\end{prop}

From Proposition~\ref{t-sigma} and the above analysis, we have readily

\begin{prop}
\label{Nice}
\begin{enumerate}
\item If  $t\circ_{\ve} [\bx]= t\circ_{-c\ve} [\bx]$, for all $t\in \mbb C^*$ and for all $[\bx]\in \M_{\zeta}(\bv,\bw)$,
then the $\mbb C^*$-action in (\ref{t1}) on $\M_{\zeta}(\bv,\bw)$ induces
a $\mbb C^*$-action on $\fS_{\zeta}(\bv, \bw)$ such that the weight of the symplectic form $\omega$ on $\fS_{\zeta}(\bv,\bw)$ is $1$ with respect to this $\mbb C^*$-action. 

\item If  $t\circ_{\ve} [\bx]= t\circ_{-c\ve} [\bx]$, for all $t\in \mbb C^*$ and $[\bx]\in \M_{\zeta}(\bv,\bw)$,
then the $\mbb C^*$-action  provides a contraction
of $\fS_1(\bv,\bw)$  to its fixed-point $\fS_1(\bv,\bw)^{\mbb C^*}$
consisting of a single point $[0]$.
\end{enumerate}
\end{prop}

The following proposition provides  compatible cases sufficient to prove our theorems.

\begin{prop}
\label{WW}
\begin{enumerate}
\item If $c=-1$, then $t\circ_{\ve} [\bx] = t\circ_{-c\ve} [\bx]$, $\forall t\in \mbb C^*$ and $[\bx]\in\M_{\zeta}(\bv,\bw)$.

\item Assume $c=1$ and   $\bw_i \bw_j=0$ if $i$ and $j$ are joined by an edge.
Let $I=I^{1} \sqcup I^{0}$ be a partition satisfying the following conditions.
\begin{itemize}
\item For all $i\in I^{0}$, we have $\bw_i=0$. 
\item For all $h\in \ve^{-1} (1)$, we have $\o (h)  \in I^{1}$ and  $\i(h) \in I^0$.
%\item $a(I^1) = I^1$ and $a(I^0) = I^0$. 
\end{itemize}
Then $t\circ_{\ve} [\bx] = t\circ_{-c\ve} [\bx]$ for all  $t\in \mbb C^*$ and $[\bx]\in\M_{\zeta}(\bv,\bw)$.
\end{enumerate}
\end{prop}

\begin{proof} 
The first statement is obvious. Let $c=1$.
It is enough to show that $g. (t\circ_{\ve} \bx) =  t\star_{-c\ve} \bx$. 
Let $\kappa_i$ be the parity of $i$, i.e., $\kappa_i=1$ if $i\in I^{1}$ and $\kappa_i =0$ if $i\in I^{0}$. 
Let $g_{\kappa} =(t^{\kappa_i} \mrm{id}_{V_i})_{i\in I}\in \G_{\bv}$.
Then we have the following computations.
\begin{align*}
g_{\kappa}. (t\circ_{\ve} x_h)  & = t^{-1} (t\circ_{\ve} x_h) =  x_h,  \ && \ \mbox{if} \ h\in \ve^{-1} (1) ,\\
g_{\kappa}. (t\circ_{\ve}x_h)  &= t (t\circ_{\ve} x_h) = t x_h, \ && \ \mbox{if} \ h \in \ve^{-1}(-1),\\
g_{\kappa} .(t\circ_{\ve} p_i)  &= t(t\circ_{\ve} p_i) = tp_i,  \ && \ \mbox{if} \ i\in I^1,\\
g_{\kappa} . (t\circ_{\ve} q_i ) &= t^{-1} (t\circ_{\ve} q_i) = q_i, \ && \ \mbox{if} \ i\in I^1.
\end{align*}
Since $p_i=0, q_i=0$ for all $i \in I^0$, the above computation shows that 
$g_{\kappa}. (t\circ_{\ve} \bx) =  t\star_{-c\ve} \bx$. 
The proof is thus finished. 
\end{proof}

\subsection{The proof of Theorems~\ref{Main-1},~\ref{Main-2} and \ref{Main-3}}

Since $\Gamma$ is a Dynkin graph, hence bipartite, so we can find a partition of $I$ such that 
the first condition in Proposition~\ref{WW} holds. Now set $\ve$ to be the unique orientation 
such that the second condition in Proposition~\ref{WW} is valid.
Since $c=1$, we see that the automorphism $a$ is compatible with the orientation $\ve$. 
In this case, the results in Proposition~\ref{Nice} are true and so Theorem~\ref{symplectic}
is applicable and from which Theorem~\ref{Main-3} follows.

In light of Proposition~\ref{8.3.4}, Theorem~\ref{Main-2}, and hence Theorem~\ref{Main-1}, follows from Theorem~\ref{Main-3}.
Note that we must show that all parabolic subgroups, up to conjugations, appear in the setting
of Proposition~\ref{8.3.4}.
But this is already  observed in Maffei's work~\cite[Theorem 8]{M05}. 

%\begin{rem}
%The condition that $r_i = r_{n-i}$ in~\cite{M05}  is equivalent to the condition $\bv = w_0*_{\bw} \bw$. 
%The arxiv version of Maffei's work contains more on this in Sections 1.3 and 1.4.
%\end{rem}

The proof of Theorems~\ref{Main-1},~\ref{Main-2} and \ref{Main-3} is finished.

\subsection{A generalization of Theorem~\ref{Main-3}}

In Proposition~\ref{WW}, there is no assumption on $\bw$ when $c=-1$, which is not stated in Theorem~\ref{Main-3},  and the above argument works 
in this more general case as well. Let us record this more general result here.

\begin{thrm}
\label{Main-4}
Let $(a, \ve^0)$ be a compatible pair of signature $c=-1$. 
Then
the fiber  $(\pi_{\sigma})^{-1} ([0])$ is Lagrangian 
in $\fS_{\zeta}(\bv, \bw)$. 
\end{thrm}

\section{Spaltenstein's examples}
\label{Example}
In this section, we discuss examples in~\cite[11.6, 11.8]{Sp82}, except 11.8 c).
We show that $X_x^P$ is Lagrangian  in all these examples, except the counterexample in~\cite[11.6]{Sp82}.
%by dimension consideration or applying Theorem~\ref{Main-2}.

\subsection{\cite[11.6]{Sp82}}
\label{11.6}
%In this subsection, we shall reproduce Spaltenstein's example in~\cite[11.6]{Sp82}
%that $X^P_x$ is not pure dimensional in general.

Let us fix a basis $\{e_i\}_{1\leq i\leq 8}$ of $\mbb C^8$.
Let $B(-,-)$ be the bilinear form defined by $B(e_i, e_j) = \delta_{i, 9-j}$ for all $1\leq i, j\leq 8$, so that
the associated symmetric matrix is the anti-diagonal identity matrix. 
Let $G=\mrm{SO}_8(\mbb C)$ be the special orthogonal group of $B(-,-)$ and $\mathfrak{so}_8(\mbb C)$ be its Lie algebra.
Let $x$ be an element of the form
\[
x=
\begin{bmatrix}
x_1 & 0 & 0 \\
0 & x_2 & 0 \\
0 & 0 & -x_1
\end{bmatrix} \
x_1 =
\begin{bmatrix}
0 & 1 \\
0 & 0
\end{bmatrix}\
x_2 =
\begin{bmatrix}
0 & 1 & 1 & 0\\
0 & 0 & 0 & -1\\
0 & 0 & 0 & -1\\
0 & 0 & 0 & 0
\end{bmatrix}
\]
Then it is clear that $x$ is of Jordan type $(1, 2^2, 3)$ and is a nilpotent element in $\mathfrak{so}_8(\mbb C)$.
Let $G/P$ be the isotropic flag variety of isotropic subspaces $F_2\subseteq F_3$ in $\mbb C^8$ such that
$\dim F_2=2$ and $\dim F_3=3$. 
Then 
the Spaltenstein variety $X_x^P$ of the triple ($\mrm{SO}_8(\mbb C), P, x)$ is the  subvariety of $G/P$ consisting of elements $(F_2\subset F_3)$ such that 
$
x(F_2) =0,
x(F_3)\subseteq F_2,
x(F_3^{\perp}) \subseteq F_3.
$
There is a partition of $X_x^P=X_3\sqcup X_2$ where 
\[
X_3= \{ (F_2\subseteq F_3)\in X^P_x | e_3\in F_2\},\
X_2= \{ (F_2\subseteq F_3)\in X^P_x | e_3\not\in F_2\}.
\]
One can check that $X_3$ and $X_2$ are irreducible of dimension $3$ and $2$, respectively. 
Indeed, for a fixed flag  $F_2$ in $X_3$, the freedom of $F_3$ is $\mrm{OGr}(1, 4)$, the Grassmannian of isotropic lines in $\mbb C^4$.
The dimension of $\mrm{OGr}(1,4)$ is  $2$, hence the dimension of $X_3$ is 3.
For a fixed flag $F_2$ in $X_2$, there is a unique flag $F_3$, i.e., $F_3 = \langle F_2, e_3\rangle$.
Thus the dimension of $X_2$ is $2$. 
%It implies readily 

So the irreducible components of $X_x^P$ are $X_3$ of dimension $3$ 
and the closure of $X_2$ in $X^P_x$ of dimension 2. 
Hence $X_x^P$ is not pure dimensional.

Let $Q$ be a parabolic subgroup such that $G/Q$ is the isotropic flag varieties 
of all  flags $F_1\subset F_2 \subset F_4$ such that $\dim F_i = i$. 
From~\cite[11.6]{Sp82},  $X_x^Q$ is  irreducible and  of dimension 3. 
Let $e_Q$ be the Richardson element associated to $Q$. Then it can be shown that
$\dim \tilde S_{e_Q,x} = 6$, hence $X_x^Q$ is Lagrangian in $\tilde S_{e_Q, x}$.
This example is not in the cases (a)-(c) in the introduction.

\subsection{\cite[11.8. a)]{Sp82}}

If $G$ is of type $A_n$ (resp. $D_n$; $E_6$; $E_7$; $E_8$), $\dim X^B_x =2$, with $B$ a Borel,  and $P$ is minimal, then
$X^P_x$ is a union of projective lines in a configuration of type $A_{n-2}$
(resp. $A_1$ or $D_{n-2}$, the last is only possible if $n\geq 5$; $A_5$; $D_6$; $E_7$). 
%In particular, $X_x^P$ is connected.  
The condition $\dim X^B_x=2$ implies that $\dim \mathcal O_x= \dim T^*G/B -4$ and $P$ is minimal implies that 
$\dim T^*G/P=\dim T^*G/B -2$. So dimension of $\tilde S_{e_P, x}$ is 2, and thus $X^P_x$ is Lagrangian in $\tilde S_{e_P,x}$.

%For $G$ of type $D_n$, the codimension 4 orbit $\mathcal O_x$ has of Jordan type $(1^3, (2n-3))$ and hence Theorem~\ref{Main-2} is applicable in this case. 

\subsection{\cite[11.8. b)]{Sp82}}
Let $G=\mrm{SO}_7(\mbb C)$,  with $x$  of type $(3, 1^4)$ and $G/P$ a maximal isotropic Grassmannian.
Then $X_x^P$ is a disjoint union of two projective lines.
By Theorem~\ref{Main-2}, $X_x^P$ is Lagrangian in $\tilde S_{e_P,x}$.

\subsection{\cite[11.8. d]{Sp82}}

Let $G=\mrm{Sp}_{4n+2}(\mbb C)$ and $G/P$ is a partial flag variety obtained from the complete flag by dropping the $(2i+1)$-th step for all $0\leq i\leq n$,
$x$ is a nilpotent of type $((2n)^2, 2^1)$.
Then $X_x^P$ is a union of $2n+1$ projective lines subject to certain conditions.
From Theorem~\ref{Main-2}, $X_x^P$ is Lagrangian in $\tilde S_{e_P,x}$.

\subsection{}
By the rectangular symmetry in~\cite{Li18}, one can produce more examples from previous subsections. 
For example, the corresponding case in Section~\ref{11.6} for $(G, P, x)$ is
$G'=\mrm{Sp}_{12}(\mbb C)$, 
$P'$ is chosen such that $G'/P'$ is isomorphic to the isotropic flag varieties of $(F_2\subset F_5)$ with  $\dim F_i =i$,
and $x'$ is of Jordan type $(2, 3^2, 4)$. Then $X^{P'}_{x'} \cong X^P_x$ is not pure dimensional.


\begin{thebibliography}{99999}\frenchspacing

\bibitem[BM83]{BM83}
W. Borho and R. MacPherson, 
{\em Partial resolutions of nilpotent varieties}, Ast\'{e}risque {\bf 101}-{\bf 102} (1983), 23–74.

%\bibitem[B08]{B08}
%J. Brundan, 
%{\em Symmetric functions, parabolic category O and the Springer fiber}, 
%Duke Math. J. 143 (2008), 41-79.

\bibitem[G08]{G08}
V. Ginzburg, 
{\em Harish-Chandra bimodules for quantized Slodowy slices}, 
Represent. Theory {\bf 13}  (2009), 236-271. 

\bibitem[G09]{G09}
V. Ginzburg,
{\em Lectures on Nakajima's quiver varieties},
in {\em Geometric methods in representation theory. I}, 145-219, 
S\'{e}min. Congr., 24-I, Soc. Math. France, Paris, 2012. 
%arXiv:0905.0686.

%\bibitem[Kh04]{Kh04}
%M. Khovanov,
%{\em Crossingless matchings and the cohomology of (n,n) Springer varieties}, 
%Communications in Contemporary Math. 6 (2004) no.2, 561-577, arXiv:math.QA/0202110.

%\bibitem[KP82]{KP82}
%H. Kraft and C. Procesi,
%{\em On the geometry of conjugacy classes in classical groups}, 
%Comment. Math. Helv. {\bf 57} (1982), 539-602.

\bibitem[Li19]{Li18}
Y. Li,
{\em Quiver varieties and symmetric pairs},
Represent. Theory, {\bf 23} (2019), 1-56.
 \href{https://arxiv.org/abs/1801.06071}{arXiv:1801.06071}.

\bibitem[L00]{L00} 
G. Lusztig, 
{\em Quiver varieties and Weyl group actions,}  
Ann. Inst. Fourier (Grenoble) {\bf 50} (2000), no. 2, 461-489. 

\bibitem[M02]{M02}
A. Maffei,
{\em A remark on quiver varieties and Weyl groups},
Ann. Scuola Norm. Sup. Pisa Cl. Sci. {\bf 5} (2002), 649-686.

\bibitem[M05]{M05}
A. Maffei,
{\em Quiver varieties of type $A$},
Comment. Math. Helv. {\bf 80} (2005), 1-27.

\bibitem[N94]{N94}
H. Nakajima,
{\em Instantons on ALE spaces, quiver varieties, and Kac-Moody algebras,}
Duke Math. J. {\bf 76} (1994), 365-416.

\bibitem[N98]{N98}  
H. Nakajima, 
{\em Quiver varieties and Kac-Moody algebras},  
Duke Math. J. {\bf 91} (1998), no. 3, 515-560. 
        
\bibitem[N03]{N03} 
H. Nakajima,
{\em Reflection functors for quiver varieties and Weyl group actions},
Math. Ann. {\bf 327} (2003),  671-721.
%\bibitem[N00]{N00}
%H. Nakajima,
%{\em Quiver varieties and finite dimensional representations of quantum affine algebras},
%JAMS {\bf 14} (2000) no. 1, 145-238.


\bibitem[Sl80]{Sl80}
P. Slowdowy,
{\em Four lectures of simple groups and singularities},
Comm. Math. Inst. Rijksuniv. Utrecht {\bf 11}, Math. Inst. Rijksuniv., Utrecht, 1980.


\bibitem[Sp76]{Sp76}
N. Spaltenstein, 
{\em The fixed point set of a unipotent transformation on the flag manifold}, 
Indagationes Mathematicae, {\bf 38} (1976), 452-456. 

\bibitem[Sp77]{Sp77}
N. Spaltenstein, 
{\em On the fixed point set of a unipotent element on the variety of Borel subgroups}, 
Topology {\bf 16} (1977), 203-204.


\bibitem[Sp82]{Sp82}
N. Spaltenstein, 
{\em Classes unipotentes et sous-groupes de Borel}, 
Lecture Notes in Mathematics, {\bf 946}.  Springer-Verlag, Berlin-New York, 1982. 

\bibitem[Spr76]{Spr76}
T.A. Springer, 
{\em Trigonometric sums, Green functions of finite groups and representations of Weyl groups},
Invent. Math. {\bf 36} (1976), 173-207.

\bibitem[St74]{St74}
R. Steinberg,  
{\em Conjugacy classes in algebraic groups}. Notes by Vinay V. Deodhar. 
Lecture Notes in Mathematics, Vol. {\bf 366}. Springer-Verlag, Berlin-New York, 1974. 

%\bibitem[V79]{V79}
%J. Vargas, 
%{\em Fixed points under the action of unipotent elements of SL, in the flag variety}, Bol. Sot.
%Mat. Mexicana {\bf 24} (1979) 1-14.

%\bibitem[W11]{W11}
%B. Webster, 
%{\em Singular blocks of parabolic category $\mathscr O$ and finite $W$-algebras}, 
%J. Pure Appl. Algebra, {\bf 215}, 2797-2804. 





\end{thebibliography}
\end{document}